\documentclass[12pt]{article}
\usepackage{fullpage}
\usepackage{amsfonts, amsmath, amssymb,latexsym,epsfig}
\usepackage{amsthm}
\usepackage{tikz}
\usetikzlibrary{graphs}
\usetikzlibrary{graphs.standard}
\usepackage{relsize}
\usepackage{verbatim}
\usepackage{enumerate}
\usepackage{dsfont}
\usepackage[skip=10pt plus1pt, indent=20pt]{parskip}

\usepackage{hyperref}
\usepackage{bbm}

\newtheorem{thm}{Theorem}[section]
\newtheorem{cor}[thm]{Corollary}
\newtheorem{lem}[thm]{Lemma}

\newtheorem{prop}[thm]{Proposition}

\newtheorem{defn}[thm]{Definition} 
 
\newtheorem{notation}[thm]{Notation}

\newcommand{\HH}{\mathcal{H}}

\newcommand{\Z}{\mathbb Z}

\newcommand{\PG}{{\rm PG}}
\newcommand{\Gauss}[2]{{\begin{bmatrix} #1 \\ #2 \end{bmatrix}_2}}
\newcommand{\qGauss}[2]{{\begin{bmatrix} #1 \\ #2 \end{bmatrix}_q}}
\newcommand{\F}{\mathbb F}
\newcommand{\x}{\overline{x}}
\newcommand{\y}{\overline{y}}

\author{Philipp Heering\footnote{Justus-Liebig Universität Gießen,
Mathematisches Institut, Arndtstraße 2,
 D-35392 Gießen, Germany,
 philipp.heering@math.uni-giessen.de}, \ 
Vladislav Taranchuk\footnote{Ghent University,
Department of Mathematics: Analysis, Logic and Discrete Mathematics, Krijgslaan 281,
B-9000 Gent, Belgium,
vlad.taranchuk@ugent.be
}}

\title{Line-parallelisms of $\text{PG}(n, 2)$ from Preparata-like codes}

\date{August 2025}

\begin{document}

\maketitle

\begin{center}
\section*{Abstract}
\end{center}
Partitions of the binary linear Hamming code into Preparata-like codes are known to induce line-parallelisms of PG$(n, 2)$. In this paper, we show that if $P$ is any Preparata-like code contained in the binary linear Hamming code $H$ of the same length, then $H$ can be partitioned into additive translates of $P$. This generalizes a result of Baker, van Lint, and Wilson who prove this fact for the class of generalized Preparata codes. We give an explicit description for line-parallelisms obtained from such a partition via crooked Preparata-like codes and establish an equivalence criterion for such line-parallelisms.

\noindent
 \textbf{Keywords:} Preparata-like codes, Crooked functions, Line-parallelisms, Equivalence \\
 \textbf{MSC (2020):} 
 51E23, %Spreads and packing problems in finite geometry
11T06, %Polynomials over finite fields 
94A60 %Cryptography 
%51E20 % Combinatorial structures in finite projective spaces

\section{Introduction}

Let $\F_q^{n+1}$ be the $(n + 1)$-dimensional vector space over the finite field with $q$ elements, $\F_q$.  For $n\geq 3$ we denote by PG$(n,q)$ the $n$-dimensional projective space, i.e. an incidence geometry whose points are the 1-dimensional subspaces of $\F_{q}^{n+1}$ and whose lines are the 2-dimensional subspaces of $\F_q^{n+1}$ with incidence being defined via inclusion. A \textit{line-spread} of PG$(n, q)$ is a partitioning of the points of PG$(n, q)$ into lines. A \textit{line-parallelism} of PG$(n, q)$ is a partitioning of the lines of PG$(n, q)$ into line-spreads.
It is known that line-spreads exist in PG$(n, q)$ if and only if $n$ is odd \cite{Hirschfeld}. On the other hand, it is a long standing open question to determine whether line-parallelisms exist in PG$(n, q)$ for any prime power $q$ and odd positive integer $n$. In this paper, we will only consider line-spreads and line-parallelisms, hence we will call them spreads and parallelisms.

Below, we give a brief survey of known parallelisms. We separate the cases $q \neq 2$ and $q = 2$. One reason for this distinction is that the results in this paper yield new constructions of parallelisms of $\PG(n, 2)$. However, a more important reason is that historically, constructions of parallelisms of $\PG(n, q)$ differ greatly between the cases $q = 2$ and $q \neq 2$. Almost all known parallelism of $\PG(n, 2)$ fall out as a corollary of the fact that the linear Hamming code ca be partitioned into translates of a Preparata-like code of the same length. 
The existence of such a partition using the classical Preparata code was first demonstrated by Zaitsev, Zinoviev, and Semakov \cite{parallelisms_1973} in 1973, and independently by Baker \cite{baker} in 1976. This was further extended by Baker, van  Lint, and Wilson \cite{BVW} in 1983 who gave a construction of the generalized Preparata codes\footnote{In this paper, a \textit{generalized Preparata code} always refers to the codes described in \cite{BVW}.} and showed that these too can be used to partition the linear Hamming code. Our review of the literature on parallelisms of $\PG(n, 2)$ and on partitions of the linear Hamming code has indicated that recent results on partitions of the Hamming code have not been incorporated into the literature on parallelisms of $\PG(n, 2)$.
The results of V. A. Zinoviev and D. V. Zinoviev \cite{zinoviev2016generalized} from 2016, for example, imply the existence of a new class of parallelisms that have not yet been considered.
We also came across some old results that seem to have been forgotten. Therefore, aside from highlighting our own contribution to the construction of new parallelisms of $\PG(n, 2)$, we make it a priority to survey all relevant results and give the reader a complete overview of the current state of the relevant research. 

\subsection{Parallelisms of $\PG(n, q)$, $q \neq 2$}

Since the 1970's, research on the existence of parallelisms of PG$(n, q)$ has been an active area. In 1973, Denniston \cite{Denniston1973} constructed the first family of parallelisms for each $q > 2$ in PG$(3, q)$ using the Klein quadric. Around the same time, Denniston \cite{Denniston19732} also gave an example of a cyclic regular parallelism of $\PG(3, 8)$. In 1974, Beutelspacher \cite{beutelspacher_old} gave a construction of parallelisms of PG$(n, q)$ when $q > 2$ and $n = 2^k - 1$ for any $k > 1$. The main idea in Beutelspacher's construction was the embedding of PG$(n, q)$ into PG$(n, q^2)$ and proceeding to construct parallelisms via induction.

In 1998, Prince \cite{Prince} enumerated (up to equivalence)  all cyclic parallelisms of $\PG(3, 5)$, there are a total of 45 of them, 2 of which are also regular parallelisms. Later in the same year, Penttila and Williams \cite{penttila1998regular} constructed cyclic regular parallelisms of $\PG(3, q)$ for each $q \equiv 2 \pmod{3}$. This was achieved by a special partitioning of the points of the Klein quadric and then using the Klein correspondence to obtain a parallelism. Several methods have been designed to obtain new parallelisms of $\PG(3, q)$ from old ones, see \cite{CosetSwitching, JohnsonPar, Pavese_parallelism_PG3q}. 

No new cases for the existence of parallelisms were resolved until 2023, when Feng and Xu \cite{Xu2023} showed that parallelisms exist in PG$(n, q)$ for $q = 3, 4, 8, 16$ and for all odd $n$. Their approach depended on proving that the existence of a certain type of parallelism of $\PG(n, q)$ can be reduced to finding a special set of spreads in $\PG(3, q)$. In \cite{Xu2023}, such a set was found for each value of $q$ listed above.

For a more detailed history of parallelisms of PG$(n, q)$, see the survey by Johnson \cite{Johnson_survey}. It's worth noting that since this survey was published in 2003, there has been a flurry of research on parallelisms of PG$(3, q)$ \cite{Pavese_parallelism_PG3q, parallelisms_of_PG35, parallelisms_of_PG34_with_regular}. In particular, we point out that recently, Pavese and Santonastaso \cite{Pavese_parallelism_PG3q} proved for $q > 2$ that $\PG(3, q)$ admits at least $\Theta(q^{q-1}q!)$ inequivalent parallelisms when $q$ is even and $\Theta(q^{2q - 3})$ inequivalent parallelisms of when $q$ is odd.

\subsection{On Preparata-like and Hamming-like codes}

A binary linear code $C$ with parameters $[m, k, d]$ is a $k$-dimensional subspace of $\F_{2}^m$ such that the Hamming distance between any two codewords in $C$ is at least $d$. The binary linear Hamming code $H_n$ is a linear code with parameters $[2^n - 1, 2^n - n - 1, 3]$. Its parity-check matrix is the $n \times (2^n - 1)$ matrix whose columns are formed by the distinct non-zero vectors of $\F_{2}^n$. A (not necessarily linear) code $C$ is called \textit{Hamming-like} if it has the same parameters as $H_n$ for some $n$. A \textit{Preparata-like} code $P_n$ is a binary code with length $2^{n} - 1$, size $2^{2^n - 2n}$ and minimum distance 5. Preparata-like codes can only exist when $n$ is even. The first infinite family of such codes was given by Preparata \cite{Preparata} in 1968. The description of these codes was subsequently simplified and generalized in 1983 \cite{BVW}. Today, there are many known examples of Preparata-like codes \cite{Kantor, crooked_2000}.
An extended Preparata-like code $EP_n$ can be obtained from a Preparata-like code $P_n$ by adding a parity-check bit, thereby increasing the length and minimum distance of the code both by 1. In a similar manner, one can obtain an extended Hamming-like code $EH_n$ from a corresponding Hamming-like code $H_n$. We remark that results in the literature are given sometimes for the extended codes and sometimes for non-extended versions. There is a natural relationship between the two, so unless stated otherwise, we give all results in terms of their implication for the non-extended version of Preparata-like and Hamming-like codes.

In \cite{parallelisms_1973}, it was shown that any Preparata-like code $P_n$ is contained in a Hamming-like code $H_n$ of the same length $m$. In particular, it was shown that $H_n = P_n \cup E$, where $E$ is the set of all codewords in $\F_2^{m}$ whose distance to $P_n$ is at least 3. 

It is known that if an extended Preparata-like code is $\Z_4$-linear (via the gray map, see \cite{Hammons}), then the corresponding extended Hamming-like code containing it is also $\Z_4$-linear as was shown by Borges, Phelps, Rif\'{a}, and Zinoviev \cite{Borges}. The linear extended Hamming code $EH_{n}$ is not $\Z_4$-linear except when $n \leq 4$ \cite{Hammons}. In fact, for all known Preparata-like codes $P_n$, exactly one of the following holds:
\begin{enumerate}
    \item $P_n$ is contained in the linear Hamming code \cite{BVW, crooked_2000}, or
    \item $EP_n$ is $\Z_4$-linear and is contained in an extended Hamming-like code which is $\Z_4$-linear \cite{Borges, Hammons, Kantor}.
\end{enumerate}

\noindent
Since the seminal results of \cite{baker, BVW, parallelisms_1973}, Preparata-like codes and their relation to the Hamming-like codes containing them, have been studied extensively. As we noted, the class of generalized Preparata codes are known to give rise to a partition of the linear Hamming code \cite{baker, BVW, parallelisms_1973}. It is also known that any extended $\Z_4$-linear Preparata-like code $EP_n$  partitions the extended $\Z_4$-linear Hamming-like code containing it by additive translates of $EP_n$ over $\Z_4$ \cite{Borges}. Furthermore, it was recently shown that it is possible to introduce a group operation on the generalized Preparata codes and the linear Hamming code so that both would be interpreted as group codes under this operation \cite{zinoviev2016generalized}. Consequently, a new and inequivalent partition of the Hamming code was given by the cosets of the generalized Preparata code in the linear Hamming code via this group operation.

In 2009, Zinoviev asked if any Preparata-like code gives rise to a partition of the corresponding Hamming-like code containing it (see \cite{heden2009open}). Our first result resolves this question in the affirmative in the case that the Hamming-like code is linear.

\begin{thm}\label{T: Partition}
    Let $P_n$ be a Preparata-like code contained in the linear Hamming code $H_n$ of the same length. Then $H_n$ can be partitioned into additive translates of $P_n$.
\end{thm}

In the next section, we describe how such partitions induce parallelisms of $\PG(n, 2)$ and give an accounting of all known parallelisms of $\PG(n, 2)$.

\subsection{Parallelisms of $\PG(n, 2)$}\label{q=2}

Let $H_n$ and $EH_n$ be the linear and extended linear Hamming codes respectively. Likewise, let $P_n$ and $EP_n$ be some Preparata-like and extended Preparata-like code. Recall that the parity-check matrix of $H_n$ has as columns precisely the set of all non-zero vectors of $\F_{2}^{n}$. Thus, any codeword $c \in H_n$ can be considered to be the characteristic vector of a subset $S_c \subset \F_{2}^{n} \setminus\{0 \}$ satisfying $\sum_{x \in S_c}x = 0$. Since the minimum weight of $H_n$ is 3, we observe that this implies that the minimum-weight codewords of $H_n$ are precisely all vectors which are characteristic vectors of the sets $\{a, b, a+ b\} \subset \F_{2}^{n}\setminus \{ 0\}$, where $a \neq b$. In other words, these are precisely the lines of $\PG(n - 1,2)$. Observe that likewise, the minimum-codewords of $EH_n$ can be identified as subsets $S \subset \F_{2}^{n}$ of cardinality 4, satisfying $\sum_{x \in S}x = 0$. Consequently, all such vectors are the characteristic vectors of some affine plane in AG$(n, 2)$.

A partition of $H_n$ into additive translates of $P_n$ also partitions the minimum-weight codewords of $H_n$, and therefore the lines of $\PG(n - 1, 2)$. Note that since each translate is a Preparata-like code, it has minimum distance $5$. Thus, any two codewords of weight 3 in the same translate must have disjoint supports. A quick computation shows that $|H| = 2^{n - 1}|P|$, which implies that in each translate the minimum weight codewords correspond to spreads of $\PG(n - 1, 2)$, and therefore, all the translates together constitute a  parallelism of $\PG(n - 1, 2)$. All of this extends naturally to $EH_n$ and $EP_n$, where one obtains that the codewords of weight 4 in each translate of $EP_n$ form a 2-design of affine planes in AG$(n, 2)$. This implies that the Steiner quadruple system arising from the points and planes of AG$(n, 2)$ is resolvable.

The paper \cite{zinoviev2016generalized}, which gives a new partition of the extended linear Hamming code into cosets of the extended generalized Preparata code, also demonstrates that the corresponding partition of the affine planes is inequivalent to those obtained in \cite{BVW}.  Furthermore, \cite[Lemma 9]{zinoviev2016generalized} implies that each part in the partition corresponds to a unique spread in $\PG(2n-1, 2)$ and so this construction actually yields a new class of inequivalent parallelisms of $\PG(2n-1, 2)$, however, to our knowledge, this implication has not been noted anywhere.  

\subsubsection{A list of  all known parallelisms of $\PG(n, 2)$}

Including our result, all known infinite families of parallelism of $\PG(n, 2)$ except for one, arise from partitions of the linear Hamming code into Preparata-like codes. The one exception is a class of parallelisms constructed by Wettl \cite{Wettl} in 1991. It can be quickly determined that the parallelisms in \cite{Wettl} are inequivalent to those given in \cite{baker, BVW}. The parallelisms induced by the partition of linear Hamming code into translates of the generalized Preparata codes \cite{BVW} are actually described by Baker \cite{Baker2} in a separate paper, which seems to have been forgotten. We note that these same exact parallelisms were rediscovered by Johnson and Montinaro \cite[Section 5]{JohnsonMontinaro} in 2012. Below we give a list of all known infinite families of parallelisms of $\PG(n, 2)$:
\begin{enumerate}
    \item The parallelisms arising from the partitioning of the Hamming code into additve translates of some Preparata-like code as shown in \cite{baker, Baker2, BVW, JohnsonMontinaro, parallelisms_1973},which is now more generally implied by Theorem \ref{T: Partition}. 
    \item The parallelisms arising from the partitioning of the linear Hamming code into cosets of the generalized Preparata codes via a group  theoretical approach \cite{zinoviev2016generalized}.
    \item The parallelisms that are explicitly constructed in \cite{Wettl}.
    %\item The explicit construction of parallelisms in \cite{Wettl}.
\end{enumerate}

\noindent
There are also a number of sporadic results known, obtained mainly by computer search, namely for $\PG(n, 2)$ where $n = 5, 7, 9$. It is shown in particular that in these cases $\PG(n, 2)$ admits parallelisms with special properties, such as orthogonality and point-cyclic transitivity \cite{Braun, Hishida, Sarmiento, Stinson}.

We remark that Theorem \ref{T: Partition} does yield a larger class of parallelisms due to the fact that the more general class of Preparata-like codes described by van Dam and Fon-der-Flaass \cite{crooked_2000} are known to be contained in linear Hamming codes. These Preparata-like codes are linked to the existence of special functions over $\F_{2^n}$ called \textit{crooked functions}. The generalized Preparata codes are known to be related to the crooked function $f(x) = x^{2^t + 1}$ over $\F_{2^n}$ where $n$ is odd and $\gcd(n, t) = 1$. In the next section we survey the known results on crooked functions and describe the corresponding Preparata-like codes. In particular, we highlight that several new infinite families of crooked functions have been discovered in recent years
which consequently will give rise to parallelisms of PG$(n, 2)$ via our construction that are inequivalent to those in \cite{baker, Baker2, BVW,  JohnsonMontinaro}.

\subsection{On crooked functions}

Crooked functions were introduced by Bending and Fon-der-Flaass \cite{Bending1998} in order to generalize the family of distance-regular graphs constructed by De Caen, Mathon, and Moorhouse \cite{Distance}.

\begin{defn}[\cite{Bending1998}] \label{crookedfunction} 
    Let $f$ be a function on $\F_{2^n}$. %(or equivalently on $\mathbb{F}_2^n$).
    The function $f$ is called crooked if it satisfies the following conditions:
    \begin{enumerate}
        \item $f(0) = 0$.
        \item $f(x) + f(y) + f(z) + f(x + y + z) \neq 0$ for any three distinct $x, y, z \in \mathbb{F}_{2^n}$.
        \item $f(x) + f(y) + f(z) + f(x + a) + f(y + a) + f(z + a) \neq 0$ for any $a \neq 0$ and any $x, y, z\in \mathbb{F}_{2^n}$.
    \end{enumerate}
\end{defn}

\noindent \textbf{Remark}: The definition we give above is the original definition of crooked functions. We note that in more recent literature, a broader class of functions is sometimes referred to as crooked \cite{bierbrauer2008crooked, kyureghyan2007crooked}. In this paper, any function called crooked strictly obeys Definition \ref{crookedfunction}.

\begin{defn}[\cite{Bending1998}]\label{D: equivalence crooked}
    Let $f$ and $f'$ be two crooked functions over $\mathbb{F}_{2^n}$.
        \begin{enumerate}
            \item  We say that $f'$ is affine equivalent to $f$ if there exists an $\alpha \in \mathbb{F}_{2^n}$ and linear permutations $L_1, L_2$ of $\mathbb{F}_{2^n}$ such that $f'(x) = L_1(f(L_2(x) + \alpha)) + L_1(f(\alpha))$.
            \item We say that $f'$ is linearly equivalent to $f$ if $f'$ is affine equivalent to $f$ with $\alpha  = 0$.
        \end{enumerate}
\end{defn}

Since their introduction, crooked functions have found many applications, including the construction of new uniformly packed codes and Preparata-like codes \cite{crooked_2000}. Godsil and Roy \cite{godsil_crooked} have shown that crooked functions are characterized by the corresponding construction of Preparata-like codes. We describe these codes below.

\begin{defn}
Let $f$ be a function over $\F_{2^n}$. Let $X \subset \F_{2^n}^*$ and let  $ Y \subset \F_{2^n}$. Denote by $\mathbbm{1}_X$ the characteristic vector of length $2^n-1$  of the set $X$ and denote by $\mathbbm{1}_Y$ the characteristic vector of length $2^{n}$ of the set $Y$. Define the code $P_f$ to be the set of all codewords of the form $(\mathbbm{1}_X, \mathbbm{1}_Y)$ with $X \subset \F_{2^n}^*$ and $Y \subset \F_{2^n}$ where $X$, $Y$ satisfy:
\begin{enumerate}
    \item $|Y|$ is even
    \item It holds that 
    $$
    \sum_{x \in X} x = \sum_{y \in Y}y.
    $$
    \item It holds that 
    $$
    \sum_{x \in X} f(x) = \sum_{y \in Y}f(y) + f\left( \sum_{y \in Y}y \right).
    $$
\end{enumerate}
\end{defn}

\noindent \textbf{Remark}: Observe that conditions (1) and (2) are precisely the conditions which imply that the codewords also belong to the linear Hamming code of length $2^{n + 1} - 1$.

\begin{prop}[\cite{godsil_crooked, crooked_2000}]
    The code $P_f$ is a Preparata-like code if and only if $f$ is crooked.
\end{prop}

Since crooked Preparata-like codes of length $2^{n+1} - 1$ are contained in the linear Hamming code of the same length, Theorem \ref{T: Partition} implies that crooked functions also give rise to parallelisms of $\PG(n, 2)$. Next, we explicitly describe how one can construct a parallelism from a crooked function. For this we identify PG$(n,2)$ with $\F_2^{n+1}$ which we identify with $\F_{2^n}\times \F_2$. 

\begin{defn} \label{D: coloring via f}
Let $f$ be a crooked function over $\F_{2^n}$.
 We define the coloring function $c_f:(\F_{2^n}\times \F_2)\times (\F_{2^n}\times \F_2) \rightarrow \F_{2^n}$ via 
    $$ c_f((x,x_1),(y,y_1))=f(x + y) + f(x) + f(y) + f(x_1y + y_1x). $$
\end{defn}

In Section \ref{Section construction}, we show that $c_f$ induces a coloring of the lines of PG$(n,2)$, where the colors are the elements of $\F_{2^n}^*$.

\begin{thm} \label{T: c_f is parallelism}
    Let $f$ be a crooked function over $\F_{2^n}$. The color classes of $c_f$ constitute pairwise disjoint spreads and the induced parallelism is denoted by $\Pi_f$.
\end{thm}

Two parallelisms $\Pi_1$ and $\Pi_2$ of PG$(n,2)$ are said to be equivalent if there exists a collineation of PG$(n, 2)$ which maps the spreads of $\Pi_1$ to the spreads of $\Pi_2$. Our third main result establishes an equivalence criterion for parallelisms coming from our construction.

\begin{thm}\label{T: main equivalence}
Let $f$ and $f'$ be crooked functions over $\F_{2^n}$ with $n>1$ odd and let $\Pi_f$ and $\Pi_{f'}$ be the parallelisms induced by $c_f$ and $c_{f'}$. 
\begin{enumerate}
    \item If $f$ and $f'$ are linearly equivalent, then $\Pi_f$ and $\Pi_{f'}$ are equivalent.
    \item Suppose further that $f$ and $f'$ are quadratic and that $n 
    > 3$. It holds that $\Pi_f$ and $\Pi_f'$ are equivalent if and only if 
    $f$ and $f'$ are affine equivalent.
\end{enumerate}
\end{thm}

Below, we give a brief survey of the known families of inequivalent crooked functions. This requires us to first take a detour through the world of almost perfect nonlinear (APN) functions which contain the class of crooked functions \cite{Bending1998}. 

A function $f$ on $\F_{2^n}$ is said to be APN if the equation 
$$
f(x + a) + f(x) = b
$$
has precisely 0 or 2 solutions in $x$ for each $a, b \in \F_{2^n}$, and $a \neq 0$. 
Observe that any crooked function is necessarily APN \cite{pott_survey}.
APN functions were introduced by Nyberg \cite{eurocrypt-1993-2628}, who demonstrated their utility in cryptography. Several notions of equivalence, such as linear equivalence, affine equivalence, (extended affine) EA-equivalence and CCZ-equivalence have been introduced in the study of APN functions and much work has gone into finding new inequivalent APN functions \cite{BCL_APN, new_apn_trivariate}. 
The definition of linear equivalence and affine equivalence for APN functions are the same as for crooked functions. EA-equivalence however is a more general notion. Two APN functions $f$ and $f'$ are EA-equivalent if there exist affine permutations $A_1, A_2$ and an affine function $A_3$ such that $ f'(x) = A_1(f(A_2(x))) + A_3(x)$. We remark that two crooked functions that are affine equivalent are necessarily also EA-equivalent as APN functions.

It is known that crooked functions are necessarily bijective \cite{Bending1998}. It follows that any quadratic APN permutation is crooked. To date, the only known crooked functions are quadratic. It has been conjectured that any crooked function must be quadratic, and this conjecture has been confirmed in some special cases \cite{bierbrauer2008crooked, kyureghyan2007crooked}. Until 2008, the only known crooked functions were precisely $f(x) = x^{2^t + 1}$ when $n$ is odd and $\gcd(n, t) = 1$. Since then, several new infinite families have been discovered \cite{BCL_APN, new_apn_trivariate}. Consequently, our construction implies the existence of inequivalent parallelisms of PG$(n, 2)$ to those from \cite{baker, BVW, JohnsonMontinaro, parallelisms_1973}. 

The following list contains the only known examples of crooked functions, all of which are quadratic APN permutations. We remark that the families (1) and (2) below have been proven to be EA-inequivalent to each other \cite{BCL_APN}. Family (3) is conjectured to be inequivalent to families (1) and (2).

\begin{enumerate}
    \item Let $n$ be odd and gcd$(t, n) = 1$. Then
    $$
    f(x) = x^{2^t + 1}
    $$ 
    is a quadratic APN permutation over $\F_{2^n}$. These correspond to the generalized Preparata codes. The functions of the form above (without restrictions on $n$ and $t$) are referred to as \textit{Gold functions} \cite{gold}. Their APN properties were observed by Nyberg \cite{eurocrypt-1993-2628}. It is known that $x^{2^t + 1}$ and $x^{2^s + 1}$ are EA-equivalent if and only if $s = t$ or $s = n - t$.
    \item Let $s$ and $k$ be positive integers such that $\gcd(6, k) = \gcd(3k, s) = 1$. Define $i = sk \pmod{3}$, $t = 3 - i$ and $n = 3k$. Let $w \in \F_{2^n}^*$ be an element of order $2^{2k} + 2^k + 1$. Then
    $$
    f(x) = x^{2^s + 1} + wx^{2^{ik} + 2^{tk + s}}
    $$
    is a quadratic APN permutation over $\F_{2^n}$. This family was constructed by Budaghyan, Carlet and Leander \cite{BCL_APN}. To our knowledge no equivalence criterion has been given which specifies the parameter sets $(s, w)$ that yield EA-equivalent functions.
\item Let $n$ be an odd positive integer not divisible by 7. Let $i$ a positive integer satisfying $\gcd(i, n) = 1$ and set $q = 2^i$. Then the function (in triviate form)
$$
f(x, y, z) = (x^{q+1} + xy^q + yz^q, xy^q + z^{q+1}, x^qz + y^{q+1} + y^q)
$$
yields a quadratic APN permutation over $\F_{2^{3n}}$ for each valid choice of $i$. Furthermore, each choice of $i$ yields an EA-inequivalent function. This class of functions was constructed recently by  Li and Kaleyski \cite{new_apn_trivariate}. The inequivalence result was given by Gao, Kan, Peng, Shi \cite[Corollary 1]{shi2025ccz}.
\end{enumerate}

\noindent \textbf{Remark}: We note that before the family (3) above was discovered in \cite{new_apn_trivariate}, there were two sporadic examples of quadratic APN permutations over $\F_{2^9}$ which were found by Beierle and Leander \cite{BierleLeander} in 2022 by computer search. These were shown to be equivalent to the specific instances of (3) when $n = 3$ \cite{new_apn_trivariate}.

We are not the first to study incidence structures in $\PG(n, q)$, or codes which are linked to APN functions. In particular, a number of topics have been explored recently in this area including the construction of linear codes, designs, and other geometric structures, see \cite{Pott_APN_planarity, Pott2}. For more information on APN functions and their connections with finite geometry, see also the survey of Pott \cite{pott_survey}.

The rest of the paper is organized as follows: In Section \ref{Section Preliminaries} we give the necessary preliminaries on APN functions and crooked functions, as well as on some notions in finite geometry regarding spreads and parallelisms. In Section \ref{Section proof of Theorem 1} we prove Theorem \ref{T: Partition}. In Section \ref{Section construction}, we give an explicit description of the line parallelisms implied by Theorem \ref{T: Partition}, thus proving Theorem \ref{T: c_f is parallelism}. This description is then utilized in Section \ref{Section equivalence} where we prove Theorem \ref{T: main equivalence}.

\section{Preliminaries}
\label{Section Preliminaries}

In this section we introduce crooked and APN functions as well as parallelisms as separate objects and collect some preliminary results.

\subsection{APN and crooked functions}

Let $f$ be an APN function over $\F_{2^n}$. A key tool which is often used in the study of these functions is the  \textit{derivative of $f$ in the direction of $a\in \F_{2^n}^*$}, denoted 
    $$
    D_a(f(x)) = f(x + a) + f(x).
    $$

It was shown in \cite{Bending1998} that crooked functions can be characterized by the image sets of their derivatives in any fixed direction.

\begin{prop}[\cite{Bending1998}]\label{P: Bending}
    Let $f$ be a function on $\F_{2^n}$ satisfying $f(0) = 0$ and let $a\in \F_{2^n}^*$.  Then $f$ is crooked if and only if the sets
    $$
    H_a = \{ D_a(f(x)) : x \in \F_{2^n}\}
    $$
    are all distinct and each $H_a$ is the complement of a hyperplane in $\F_{2^n}$. 
\end{prop}

\noindent \textbf{Remark}: Consider a crooked function $f$, let $a\in \F_{2^n}^*$ and let $H_a$ be as in the proposition above. As $f(a)\in H_a$, we have that $H_a$ is the complement of the hyperplane defined by $H_a + f(a)$.

An important property of crooked functions, which can be proven quickly from the definition, is that any crooked function must be bijective. Our construction of the parallelisms of $\PG(n, 2)$ using crooked functions will rely on this fact, among other things.

\begin{lem}[\cite{Bending1998}]  \label{L: f is permutation}
Let $f$ be a crooked function over $\F_{2^n}$. Then $f$ is a permutation.
\end{lem}

For every  APN function $f$, there is an associated symmetric bivariate function which will also play a key role in our construction.

\begin{defn}
For an APN function $f$ we define the bivariate function $B_f:\F_{2^n}\times \F_{2^n}\rightarrow \F_{2^n}$ by
    $B_f(x,y)=f(x)+f(y)+f(x+y)$.
\end{defn}

If $f$ is quadratic, then it follows that $B_f(x, y)$ is bilinear. It has in fact been shown that if $f$ and $f'$ are two quadratic APN functions for which $B_f = B_{f'}$, then  $f$ and $f'$ are EA-equivalent (see Lemma 5 in \cite{Yves}).

\begin{lem} \label{L: well defined on H}
    Let $f$ be a function over $\mathbb{F}_{2^n}$.  Then $f$ is APN if and only if $B_f(a, x) = b$ has zero or two solutions in $x$ for each $a \in \F_{2^n}^*$ and $b \in \F_{2^n}$.
\end{lem}

\begin{proof}
    Recall that $f$ is an APN function if and only if for each $a \neq 0$, we have that $f(x) + f(x + a) =b$ has zero or two solutions for any $b\in \F_{2^n}$. 
    By defining $b':=f(a)+b$, we have that $f$ is an APN function if and only if
 $f(x) + f(x + a) = b'$ has zero or two solutions.
    Equivalently, $f$ is an APN function if and only if for each $a \neq 0$ we have that
    $$
   B_f(a, x) = f(x) + f(x + a) + f(a) = b
    $$
    has zero or two solutions. 
\end{proof}

\subsection{Spreads and Parallelisms}

In this preliminary subsection, we collect results on parallelisms.
For integers $q,a,b\in\mathbb{Z}$ with $q\ge 2$ we define the Gaussian coefficient
 \begin{align*}
 \qGauss{b}{a}:=
 \begin{cases}
\displaystyle \prod_{i=1}^a\frac{q^{b-a+i}-1}{q^i-1} & \text{if $0\le a\le b$,}
\\
0 & \text{otherwise.} \end{cases}
\end{align*}
Unless specified otherwise we can always assume that $n$ is an odd, positive integer.

Recall from the introduction that a spread of $\PG(n,q)$ is a partition of the points into lines.
Furthermore, a parallelism $\Pi$ of $\PG(n,q)$ is a partition of the lines into spreads.

\begin{defn} \label{D: equivalence parallelisms}
    Two parallelisms $\Pi_1$ and $\Pi_2$ of $\PG(n,q)$ are equivalent if there exists a collineation $\kappa$ on $\PG(n,q)$ that induces a one-to-one correspondence between the spreads of $\Pi_1$ and the spreads of $\Pi_2$.
\end{defn}

Let $P$ be a fixed point of $\PG(n,q)$. No two distinct lines incident with $P$ can be in one spread, but every spread has to contain a line that is incident with $P$. This implies the following: If $\Pi$ is a parallelism of $\PG(n,q)$ the number of spreads in $\Pi$ is equal to the number of lines incident with $P$, which is $\qGauss{n}{1}$.

The size of a spread can be calculated as the quotient of the number of points in $\PG(n,q)$ and the number of points on each line. We conclude that a spread has size 
$$\frac{\qGauss{n+1}{1}}{q+1}=q^{n-1}+...+q^4+q^2+1.$$

\begin{lem} \label{L: spread inside hyperplane size}
For any spread and any hyperplane of $\PG(n,q)$, the number of lines of the spread contained inside the hyperplane is $q^{n-3}+...+q^4+q^2+1$.
\end{lem}

\begin{proof}
Let $H$ be a hyperplane and let $S$ be a spread of $\PG(n,q)$.
All points of $H$ are covered by lines of the spread.
Each line of $S$ contained inside $H$ covers $q+1$ points of $H$ and each line of $S$ not contained inside $H$ covers $1$ point of $H$.
Let $s$ be the number of lines of $S$ that are contained in $H$. 
Then $|H|=s(q+1)+|S|-s$ and so 
\begin{align*}
    s&=\frac{|H|-|S|}{q}\\
    &=\frac{\qGauss{n}{1}-\qGauss{n+1}{1}/(q+1)}{q}\\
    &=\frac{ (q^{n-1}+...+q^2+q+1)-(q^{n-3}+...+q^4+q^2+1)}{q}\\
    &=\frac{ q^{n-1}+...+q^5+q^3+q}{q}\\
    &=|S|-q^{n-1}. \qedhere
\end{align*}
\end{proof}

\section{Partitioning the Hamming code}
\label{Section proof of Theorem 1}

In this section we prove Theorem \ref{T: Partition}. The proof is fairly straightforward and follows the proof of this fact for generalized Preparata codes \cite{BVW} quite closely. For convenience we restate our first main result.

\noindent
\textbf{Theorem \ref{T: Partition}.} \emph{Let $P_n$ be a Preparata-like code contained in the linear Hamming code $H_n$ of the same length. Then $H_n$ can be partitioned into additive translates of $P_n$.}

\begin{proof}%[Proof of Theorem \ref{T: Partition}]
    Suppose that $P_n$ is contained in $H_n$. 
    First, assume that the  zero codeword is in $P_n$.
    Since $H_n$ is a binary linear Hamming code, we can identify the coordinates of the codewords with the non-zero vectors of $\F_{2}^{n}$. Let $\{ e_1, e_2, \dots, e_{n} \}$ be the usual standard ordered basis for $\F_{2}^{n}$ and define $T := \text{span}\{e_2, \dots, e_{n} \} \setminus \{0 \}$. Let $\alpha \in T$  and define $t_\alpha$ to be the characteristic vector (with coordinates ordered in the same way as the codewords in $H_n$) of the set $\{ e_1, \alpha, e_1 + \alpha \}$. 
    As the sum of the elements in $\{ e_1, \alpha, e_1 + \alpha \}$ equals $0$, we have that $t_\alpha \in H_n$ for any such $\alpha$. Furthermore, we observe that $t_\alpha$ has weight 3, and $t_\alpha + t_\beta$ has weight 4 if $\alpha \neq \beta$, as $t_\alpha$ and $t_\beta$ have a common element in their support. 
    
    Let $P_n$ be the Preparata-like code contained in $H$ and define the set $P_n+ t_\alpha = \{ p + t_\alpha : p \in P_n\}$. We claim that the sets
    \begin{align}\label{E: Partition}
    \{ P_n \} \cup \{ P_n + t_\alpha : \alpha \in T \}    
    \end{align}
    form a partition of $H_n$. Observe that since $P_n \subset H_n$, $t_\alpha \in H_n$, and $H_n$ is linear, we have $P_n+ t_\alpha \subset H_n$. As $|P_n +t_\alpha| = |P_n| = 2^{2^{n}-2n}$ and $|H_n| =2^{2^{n} - n - 1}  = 2^{n-1}|P_n| = (|T| + 1)|P_n|$, it suffices to show that the sets (\ref{E: Partition}) are disjoint.
    
    Recall that $P_n$ has minimum distance 5. Since $t_\alpha$ is a vector of weight 3, then $P_n$ and $P_n+t_\alpha$ are necessarily disjoint.  If $\alpha$ and $\beta$ are distinct elements in $T$ with  $(P_n+t_\alpha) \cap (P_n + t_\beta) \neq \emptyset$, then there exists codewords $c, d \in P_n$ such that $c + t_\alpha = d + t_\beta$ which implies $c+d = t_\alpha + t_\beta$. We already noted that $t_\alpha + t_\beta$ has weight 4, which would imply that $c + d$ has weight $4$. This yields that the distance between $c, d \in P_n$ is 4, a contradiction. Thus, (\ref{E: Partition}) is a partition of $H_n$ into additive translates of $P_n$.

    If the zero codeword is not in $P_n$, we may add any codeword $c\in P_n$ to $P_n$ to obtain the code $P_n'=P_n+c$. Now $P_n'$ contains the zero codeword and is still contained in $H_n$, since $H_n$ is linear. As $P_n'$ is an additive translate of $P_n$ the arguments above all work the same way if one replaces $P_n$ by $P_n'$.
\end{proof}

\begin{cor}\label{C: parallelisms}
    Let $P_n$ be a Preparata-like code which embeds into a linear Hamming code $H_n$ of the same length. The partitioning of $H_n$ into translates of $P_n$ induces a line parallelism of $\PG(n-1, 2)$.
\end{cor}

\section{The construction}
\label{Section construction}

Corollary \ref{C: parallelisms} implies the existence of parallelisms coming from any Preparata-like code contained in a linear Hamming code of the same length. 
By definition, the Preparata-like codes constructed in \cite{crooked_2000} using crooked functions are contained in a linear Hamming code of the same length.
In this section, we prove Theorem \ref{T: c_f is parallelism} thus giving an explicit description of the parallelisms which arise from a partitioning of the Hamming code via the corresponding crooked Preparata-like codes.

Identify $\PG(n,2)$ with $\F_2^{n+1}$ and $\F_2^{n+1}$ with $\F_{2^n}\times \F_2$.
In particular, we identify the non-zero vectors $(x,x_1)\in \F_{2^n}\times \F_2$ with points in $\PG(n, 2)$. Recall the coloring function $c_f$, for convenience we restate the definition.

\noindent
\textbf{Definition \ref{D: coloring via f}.}  Let $f$ be a crooked function over $\F_{2^n}$.
 We define the coloring function $c_f:(\F_{2^n}\times \F_2)\times (\F_{2^n}\times \F_2) \rightarrow \F_{2^n}$ via 
    $$ c_f((x,x_1),(y,y_1))=f(x + y) + f(x) + f(y) + f(x_1y + y_1x). $$

Clearly the function $c_f$ is symmetric. Next, we study when the function yields $0$.

\begin{lem} \label{L: c_f is 0}
    For a crooked function $f$ over $\F_{2^n}$, we have $c_f((x,x_1),(y,y_1))=0$ if and only if $(x,x_1)=(0,0)$ or $(y,y_1)=(0,0)$ or $(x,x_1)=(y,y_1)$.
\end{lem}
\begin{proof}
    If $(x,x_1)=(0,0)$ or $(y,y_1)=(0,0)$ or $(x,x_1)=(y,y_1)$ it follows directly from the definition that  $c_f((x,x_1),(y,y_1))=0$. So let us assume that $c_f((x,x_1),(y,y_1))=0$. We have $x_1,y_1\in \F_2$, so in view of the symmetry there are three cases to consider.

    \underline{Case 1:} $x_1=y_1=0$.\\
    In this case $0=c_f((x,x_1),(y,y_1))=f(x)+f(y)+f(x+y)$ is equivalent to $f(x+y)+f(x)=f(y)$. Since $f$ is crooked, and therefore APN, it follows that for a fixed $y$ there are precisely $0$ or $2$ solutions in $x$. The solutions are $x=y$ and $x=0$.

    \underline{Case 2:} $x_1=1$, $y_1=0$.\\
    In this case $0=c_f((x,x_1),(y,y_1))=f(x)+f(x+y)$. If the elements $x$ and $x+y$ are distinct, then so are $f(x)$ and $f(x+y)$, as $f$ is a permutation (see Lemma \ref{L: f is permutation}). Therefore, we have $x=x+y$ which yields $y=0$.

    \underline{Case 3:} $x_1=y_1=1$.\\
      In this case $0=c_f((x,x_1),(y,y_1))=f(x)+f(y)$. If the elements $x$ and $y$ are distinct, then so are $f(x)$ and $f(y)$, as $f$ is a permutation. Therefore, we have $x=y$.
\end{proof}

\begin{lem} \label{L: c_f is coloring of lines}
For a crooked  function $f$ over $\F_{2^n}$ the function $c_f$ induces a coloring on the lines of $\PG(n,2)$, where the colors are the elements of $\F_{2^n}^*$ 
\end{lem}

\begin{proof}
    Let $(x,x_1),(y,y_1)\in \F_{2^n}\times \F_2$ be distinct and nonzero. The line spanned by $(x,x_1)$ and $(y,y_1)$ (written as $\langle (x,x_1),(y,y_1)\rangle$) contains three points and the third point is $(x+y,x_1+y_1)$. We have to show $c_f((x,x_1),(y,y_1))=c_f((x,x_1),(x+y,x_1+y_1))=c_f((y,y_1),(x+y,x_1+y_1))$. We only show the first equality and leave the latter to the reader.
    \begin{align*}
        c_f((x,x_1),(x+y,x_1+y_1))&=f(x)+f(x+y)+f(x+(x+y))+f(x_1(x+y)+(x_1+y_1)x)\\
        &= f(x)+f(x+y)+f(y)+f(x_1x+x_1y+x_1x+y_1x)\\
        &= f(x)+f(x+y)+f(y)+f(x_1y+y_1x)\\
        &= c_f((x,x_1),(y,y_1)).
        \qedhere
    \end{align*}
\end{proof}

\begin{notation}
    Let $\HH \subset \PG(n, 2)$ be the hyperplane identified with  $\F_{2^n}\times \{ 0\}$.
\end{notation}

As can be observed in the definition of the coloring function $c_f$, this hyperplane will play a special role. We restate our second main result here for convenience.

\noindent
\textbf{Theorem \ref{T: c_f is parallelism}.}  \emph{Let $f$ be a crooked function over $\F_{2^n}$. The color classes of $c_f$ constitute pairwise disjoint spreads and the induced parallelism is denoted by $\Pi_f$.}

\begin{proof}
First, we show that the map $c_f$ assigns distinct colors to intersecting lines, so let $\ell_1$ and $\ell_2$ be two distinct, intersecting lines of $\PG(n,2)$.
In view of $\HH$, we consider four possibilities for the positions of $\ell_1$ and $\ell_2$. As $\HH$ is a hyperplane, every line of $\PG(n,2)$ either contains a single point of $\HH$ or is entirely contained in $\HH$.

\underline{Case 1:} Assume that $\ell_1,\ell_2$ are not contained in $\HH$, but $\ell_1\cap \ell_2\in \HH$.\\
In this case we can find $x,y,z\in \F_{2^n}$, such that $\ell_1=\langle (x,0),(y,1) \rangle$ and $\ell_2=\langle (x,0),(z,1) \rangle$. 
If $c_f$ assigns $\ell_1$ and $\ell_2$ the same color we have 
$$ c_f((x,0),(y,1))=f(y)+f(x+y)=f(z)+f(x+z)= c_f((x,0),(z,1)). $$
For fixed $x$ and $z$ the fact that $f$ is an APN function implies that there are $0$ or $2$ solutions in $y$. The solutions are $y=z$ and $y=x+z$ and both would imply $\ell_1=\ell_2$, so this is a contradiction.

\underline{Case 2:} Assume that $\ell_1,\ell_2$ are not contained in $\HH$, and also $\ell_1\cap \ell_2\notin \HH$.\\
In this case we can find $x,y,z\in \F_{2^n}$, such that $\ell_1=\langle (x,1),(y,0) \rangle$ and $\ell_2=\langle (x,1),(z,0) \rangle$. 
If $c_f$ assigns $\ell_1$ and $\ell_2$ the same color we have 
$$ c_f((x,1),(y,0))=f(x)+f(x+y)=f(x)+f(x+z)= c_f((x,1),(z,0)). $$
This yields $f(x+y)=f(x+z)$. 
By Lemma \ref{L: f is permutation}, we have that $f$ is a permutation, which yields $y=z$, a contradiction.

\underline{Case 3:} Assume that $\ell_1$ is contained in $\HH$, whereas $\ell_2$ is not contained in $\HH$.\\
In this case $\ell_1\cap \ell_2\in \HH$ and  we can find $x,y,z\in \F_{2^n}$, such that $\ell_1=\langle (x,0),(y,0) \rangle$ and $\ell_2=\langle (x,0),(z,1) \rangle$. 
If $c_f$ assigns $\ell_1$ and $\ell_2$ the same color we have 
$$ c_f((x,0),(y,0))=f(x)+f(y)+f(x+y)=f(z)+f(x+z)= c_f((x,0),(z,1)). $$
This yields $f(x) = f(x + y) + f(y) + f(x + z) + f(z)$ which stands in contradiction to $f$ being crooked (see part (3) of Definition \ref{crookedfunction} with $x = a$ ).

\underline{Case 4:} Assume that $\ell_1, \ell_2$ are contained in $\HH$.\\
In this case $\ell_1\cap \ell_2\in \HH$ and  we can find $x,y,z\in \F_{2^n}$, such that $\ell_1=\langle (x,0),(y,0) \rangle$ and $\ell_2=\langle (x,0),(z,0) \rangle$. 
If $c_f$ assigns $\ell_1$ and $\ell_2$ the same color we have 
$$ c_f((x,0),(y,0))=f(x)+f(y)+f(x+y)=f(x)+f(z)+f(x+z)= c_f((x,0),(z,0)). $$
This yields  $f(y)+f(x+y)=f(z)+f(x+z)$. For fixed $x$ and $z$ the fact that $f$ is an APN function implies that there are $0$ or $2$ solutions in $y$. The solutions are $y=z$ and $y=x+z$ and both would imply $\ell_1=\ell_2$, so this is a contradiction.

We have shown that $c_f$ assigns distinct colors to intersecting lines.
Finally, recall that the number of spreads in a parallelism of $\PG(n,2)$ is  $\Gauss{n}{1}$, which is $2^n-1$.
On the other hand, the number of elements in $\F_{2^n}^*$ is also $2^n-1$. In view of Lemma \ref{L: c_f is 0}, this proves the statement.
\end{proof}

\section{The equivalence problem}
\label{Section equivalence}

Our goal in this section is to prove Theorem \ref{T: main equivalence}.
First, we discuss some prerequisites regarding collineations of projective spaces.

A collineation of a projective space is an permutation of the points of said space that maps collinear points to collinear points. For a thorough discussion of collineations see \cite{Cameron}. 
The fundamental theorem of projective geometry states that every collineation of $\PG(n,q)$ can be identified with a bijective, semilinear map on $\F_q^{n+1}$ \cite{Cameron}.
The only automorphism of $\F_2$ is the identity, so all semilinear maps of $\F_2^{n+1}$ are linear.
Finally, we remark that every linear map on $\F_2^{n+1}$ can be represented by an matrix in $\F_2^{(n+1)\times (n+1)}$. 

Let $\kappa$ be a collineation of $\PG(n,2)$ and let $M\in \F_2^{(n+1)\times (n+1)}$ be the associated invertible matrix. Now, let $A\in \F_2^{n\times n}$, as well as $\alpha\in \F_2^{n}$, and $B\in \F_2^{1\times n}$, and $\beta \in \F_2$, such that we can write $M$ in the block form
$$ M= \begin{pmatrix}
    A & \alpha \\
    B & \beta\\
\end{pmatrix}. $$
Let us consider a vector $\overline{x}$ of $\F_2^{n+1}$. Again we use block notation to write $\overline{x}=(x,x_1)^\top$,  where $x_1$ is the block that contains only the last coordinate and $x$ is the block that contains the first $n$ coordinates. Using this notation it may be clear that $M\overline{x}=(Ax+\alpha x_1,Bx+\beta x_1 )$.\\
As we identify $\F_2^n$ with $\F_{2^n}$ we understand $x,\alpha$ as elements of $\F_{2^n}$. By fixing any choice of basis for $\F_{2^n}$ over $\F_2$, we understand $A$ as a linear map on $\F_{2^n}$, and $B$ as a linear functional $\F_{2^n}\rightarrow \F_2$. Using $Ax$ instead of $A(x)$ and $Bx$ instead of $B(x)$ it is therefore sensible to write 
$$\kappa((x,x_1))=(Ax+\alpha x_1,Bx+\beta x_1 )$$
for a vector $(x,x_1)\in \F_{2^n}\times \F_2$.
For the remainder of this section we use $c_f \kappa$ to denote the map 
 $c_f \kappa: (\F_{2^n}\times \F_2) \times  (\F_{2^n}\times \F_2) \rightarrow \F_{2^n}$, $((x,x_1)),(y,y_1))\mapsto c_f(\kappa((x,x_1)),\kappa((y,y_1)))$.

\noindent \textbf{Remark}:
    In view of Definition \ref{D: equivalence crooked} and Definition \ref{D: equivalence parallelisms}, we see that two crooked functions $f$ and $f'$ over $\F_{2^n}$ give rise to equivalent parallelisms via $c_f$ and $c_{f'}$ if there exists a permutation $\sigma$ of $\F_{2^n}$ with $\sigma(0)=0$ and a collineation $\kappa$ of $\PG(n,2)$ such that $\sigma c_{f'}=c_f\kappa$.

It is straight forward to prove part 1 of Theorem \ref{T: main equivalence}.

\begin{lem}
    Let $f$ and $f'$ be linearly equivalent. Then $\Pi_f$ and $\Pi_{f'}$ are equivalent.
\end{lem}

\begin{proof}
    If $f$ and $f'$ are linearly equivalent, there exist linear permutations $L_1,L_2$ of $\F_{2^n}$ such that $f'=L_1fL_2$.
    For two points $(x,x_1),(y,y_1)\in \F_{2^n}\times \F_2$, this implies the following
    \begin{align*}
        c_{f'}((x,x_1),(y,y_1))&= f'(x)+f'(y)+f'(x+y)+f'(x_1y+y_1x)\\
                 &= L_1(f(L_2x))+L_1(f(L_2y))+L_1(f(L_2(x+y)))+L_1(f(L_2(x_1y+y_1x)))\\
                &= L_1(f(L_2x))+L_1(f(L_2y))+L_1(f(L_2x+L_2y))+L_1(f(x_1L_2y+y_1L_2x))\\
                &= L_1 \bigl( f(L_2x)+f(L_2y)+f(L_2x+L_2y)+f(x_1L_2y+y_1L_2x) \bigr)\\
                &= L_1 \bigl( c_{f}((L_2x,x_1),(L_2y,y_1)) \bigr).
    \end{align*}
    By defining $\sigma=L_1^{-1}$ and $\kappa$ via $\kappa((x,x_1))=(L_2x,x_1)$, it follows that $\sigma c_{f'}=c_f\kappa$.
\end{proof}

Our last goal is to show part 2 of Theorem \ref{T: main equivalence}. This is more involved and the proof is split up into several lemmas. Recall that we denote by $\HH$ the hyperplane of $\PG(n,2)$ that is identified with $\F_{2^n}\times \{ 0\}$.

\begin{lem}\label{L: images}
Let $n>1$ be an odd integer.
    For $\overline{x}\in \HH$ the image of $c_f(\overline{x},.): \HH \rightarrow \F_{2^n}$ is an $(n-1)$-dimensional subspace of $\F_{2^n}$ that is unique to $\overline{x}$.
\end{lem}

Note that $\overline{x}, \overline{y} \in \HH$ are identified with $(x,0), (y,0)$. Thus, we have that $c_f(\overline{x}, \overline{y}) = f(x + y) + f(x) + f(y)$ and the lemma above is directly implied by Proposition \ref{P: Bending} and the remark following it.

\begin{lem} \label{L: geometric configuration}
Let $f$ be a quadratic crooked function.  Let $\mathcal{C}$ be a subspace of $\PG(n,q)$ with codimension $2$, contained in $\HH$. If $n>3$ there is a spread $S$ contained inside the the parallelism $\Pi_f$ defined by $c_f$ that satisfies the following property: Not all lines of $S$ that are contained in $\HH$ are in $\mathcal{C}$, but at least one is.
\end{lem}

\begin{proof}
    We assume that there is no spread inside the parallelism $\Pi_f$ that has the desired property. This implies that if a spread $S$ has a line contained inside $\mathcal{C}$, then all the lines of $S$ contained inside $\HH$ are already contained inside $\mathcal{C}$. 
    By Lemma \ref{L: spread inside hyperplane size} there are $|S|-2^{n-1}=2^{n-3}+\ldots + 2^4+2^2+1$ lines of $S$ in $\HH$. This is precisely the size of a spread in PG$(n-2,2)$. So, in particular $\Pi_f$ induces a parallelism $\Pi_f(\mathcal{C})$ on $\mathcal{C}$.
    This implies that for a fixed $\x \in \mathcal{C}$, the image set $\{ c_f(\x, \y): \y \in \mathcal{C}\}$ is the same, call it $X$.
    As $f$ is quadratic, $c_f$ is bilinear on $\HH \times \HH$, hence the image $\{ c_f(\x, \y): \y \in \mathcal{C}\}$ is a subspace of $\F_{2^n}$. A point in $\mathcal{C}$ is incident with $\Gauss{(n-1)-1}{2-1}=\Gauss{n-2}{1}$ lines in $\mathcal{C}$, which is the number of nonzero elements in an $(n-2)$-dimensional subspace in $\F_{2^n}$. Hence $X$ is an $(n-2)$-dimensional subspace of $\F_{2^{n}}$.

    On the other hand, let $\x \in \mathcal{C}$ and define $\HH_{\x} = \{ c_f(\x, \y): \y \in \HH \}$. The discussion above implies that for any $a, b \in \mathcal{C}$ we have $\HH_a \cap \HH_b = X$. But Lemma \ref{L: images} states that the hyperplane $\HH_a$ is distinct for each $a$. Note that there are precisely $\Gauss{n-(n-2)}{(n-1)-(n-2)}=\Gauss{2}{1}=3$ hyperplanes of $\F_{2^n}$ which can pairwise intersect in the same codimension 2 subspace of $\F_{2^n}$. This would imply that the number of points in $\mathcal{C}$ is 3, which cannot be since $n > 3$.
\end{proof}

For $n=3$, Lemma \ref{L: spread inside hyperplane size} implies that a spread has only one line in a fixed hyperplane.
Therefore, the condition $n>3$ in Lemma \ref{L: geometric configuration} is necessary and this is the reason why we need $n>3$ throughout this section and in Theorem \ref{T: main equivalence}. 
By Lemma \ref{L: c_f is coloring of lines} the map $c_f$ assigns each line of $\PG(n,2)$ a color. Therefore, we will write $c_f(\ell)$ to denote the color assigned to the line $\ell$ of $\PG(n,2)$.

\begin{lem} \label{L: K(H)=H}
    Let $f$ and $f'$ be quadratic crooked functions over $\F_{2^n}$ with $n>3$ and let $\sigma$ be a permutation of $\F_{2^n}$ with $\sigma(0)=0$. Furthermore, let $\kappa$ be a collineation of $\PG(n,2)$ such that
    $ \sigma c_f=c_{f'}\kappa. $
    Then $\kappa(\HH)=\HH$.
\end{lem}

\begin{proof}
    Let $\Pi_f$ and $\Pi_{f'}$ be the parallelisms defined by $c_f$ and $c_{f'}$ respectively.
    Recall that given a quadratic function $f$ over $\F_{2^n}$, we have a corresponding symmetric bilinear function $B_f(x, y) = f(x + y) + f(x) + f(y)$. Observe that given a coloring function $c_f$, we obtain that $c_f((x, 0), (y, 0)) = B_f(x, y) = f(x + y) + f(x) + f(y)$. Thus, restricting $c_f$ to $\HH$, we recover $B_f$.

    Suppose for a contradiction that $f$ and $f'$ are quadratic APN permutations for which there exists a permutation $\sigma$ of $\F_{2^n}$ and collineation $\kappa$ of $\PG(n, 2)$ such that  $\sigma c_f=c_{f'}\kappa$ and $\kappa(\HH) \neq \HH$. This yields $\kappa^{-1}(\HH) \neq \HH$. Clearly $c_{f'}=c_{f'} \kappa \kappa^{-1}$ and so as $c_{f'}$ is bilinear on $\HH$, we have that $c_{f'}\kappa$ is bilinear on $\kappa^{-1}(\HH)$. Therefore, $\sigma c_f$ is also bilinear on $\kappa^{-1}(\HH)$. 
    
    Set $\mathcal{C} = \HH \cap \kappa^{-1}(\HH)$ and observe that $\mathcal{C}$ has codimension $2$ as $\HH \neq\kappa^{-1}(\HH)$. By Lemma \ref{L: geometric configuration}, there is a spread $S$ in $\Pi_f$ that has a line $\ell_1$ in $\mathcal{C}$ but also a line $\ell_2$ in $\kappa^{-1}(\HH)$ that is not contained in $\mathcal{C}$. Now take any line $k$ not contained in $\HH$ which intersects $\ell_1$ and $\ell_2$. Denote $\ell_1 \cap k = P_1$ and $\ell_2 \cap k = P_2$.
    Clearly such a line $k$ belongs to a spread distinct from $S$. Consider now the planes $\pi_1 = \langle \ell_1, k\rangle$ and $\pi_2 = \langle \ell_2, k\rangle$. Let $m_1$ be the line in $\pi_1$ distinct from $\ell_1$ and $k$, but which contains $P_1$. Likewise let $m_2$ be the line in $\pi_2$ which is distinct from $\ell_2$ and $k$, but contains $P_2$. We note that all of these points, lines, and planes are contained in $\kappa^{-1}(\HH)$. 

    We claim that $m_1$ and $m_2$ belong to the same spread of the parallelism defined by $c_{f'}\kappa = \sigma c_f$ but to distinct spreads of $\Pi_f$, which is a contradiction. Let $x,y,z,w\in \F_{2^n}$, such that $\ell_1 = \langle (x, 0), (y, 0)\rangle$, $\ell_2 = \langle (z, 0), (w, 1) \rangle$ and $k = \langle (w, 1), (x, 0) \rangle$. Thus, $m_1 = \langle (x, 0), (w+y, 1) \rangle$ and $m_2 = \langle (w, 1), (z+x,  0) \rangle$. Since $\sigma c_f$ is bilinear on $\kappa^{-1}(\HH)$, we have
    \begin{align*}
        \sigma c_f (m_1) &= \sigma c_f((x, 0), (w+y, 1))  \\
        &= \sigma c_f((x, 0), (w, 1)) + \sigma c_f((x, 0), (y, 0)) \\
        & =  \sigma c_f(k)+ \sigma c_f (\ell_1) .
    \end{align*}
    Likewise, we have
    \begin{align*}
        \sigma c_f (m_2) &= \sigma c_f((w, 1), (z+x, 0))  \\
        &= \sigma c_f((w, 1), (x, 0)) + \sigma c_f((w, 1), (z, 0)) \\
        & = \sigma c_f(k)+\sigma c_f (\ell_2) .
    \end{align*}
    We assumed that $\ell_1$ and $\ell_2$ come from the same spread of $\Pi_f$, this yields $\sigma c_f(\ell_1) = \sigma c_f(\ell_2)$ which implies $\sigma c_f(m_1) = \sigma c_f(m_2)$. Hence $c_f(m_1)=c_f(m_2)$, which yields that $m_1$ and $m_2$ are in the same spread in the parallelism $\Pi_f$.

    Now, let us run these computations again on $c_f$, demonstrating that $m_1$ and $m_2$ cannot belong to the same spread of $\Pi_f$:
    \begin{align*}
        c_f (m_1) &= c_f((x, 0), (w+y, 1))  \\
        &= B_f(x, w+y) + f(x)   \\
        &= B_f(x, w) + f(x) + B_f(x, y) \\
        &= c_f((x, 0), (y, 0)) + c_f((x, 0), (w, 1))\\
        & = c_f(k)+c_f (\ell_1).
    \end{align*}
    On the other hand
    \begin{align*}
        c_f (m_2) &= c_f((w, 1), (z+x, 0))  \\
        &= B_f(w, x+z) + f(z+x)   \\
        &= B_f(w, x) + B_f(w, z)+ f(x + z)  \\
        &= f(w)+f(x)+f(w+x)  + f(w)+f(z)+f(w+z)+ f(x + z) \\
        & = c_f((x, 0), (w, 1)) +f(x)+ c_f((z, 0), (w, 1)) + f(z)+ f(x + z)\\
        & =  c_f(k)+c_f (\ell_2) + f(x) + f(z) + f(x + z).
    \end{align*}
    If $c_f(m_1) = c_f(m_2)$, then it must be that $f(x) + f(z) + f(x + z) = 0$ and so $x = 0$, $z = 0$ or $x = z$. All of which are contradictions, since $x,z\in\F_{2^n}^*$ and $\ell_2\neq k$. Thus it must be that $\kappa(\HH) = \HH$. 
\end{proof}

At the beginning of this section we stated that we can understand a collineation $\kappa$ via $\kappa((x,x_1))=(A(x)+\alpha x_1,B(x)+\beta x_1 )$, where $A$ and $B$ are linear. Now let $f$, $f'$, $\sigma$ and $\kappa$ be as in Lemma \ref{L: K(H)=H}. 
As $\kappa(\HH)=\HH$, there exists a linear function $A$ on $\F_{2^n}$ such that $\kappa((x,0))=(Ax,0)$ for all $x\in \F_{2^n}^*$. We want to deduce from this that if $\sigma c_f=c_{f'}\kappa$, then $\sigma$ is linear. Before we manage to do so we need the following technical lemma.

\begin{lem} \label{L: intersecting lines in H}
Let $f$ be a quadratic crooked function and let $\Pi_f$ be the associated parallelism.
For any two spreads $S_1,S_2\in \Pi_f$ there are intersecting lines $\ell_1\in S_1$ and $\ell_2\in S_2$ with $\ell_1,\ell_2\subset \HH$.
\end{lem}

\begin{proof}
In Lemma \ref{L: images} we have already shown that for $\overline{x}\in \HH$ the image $c_f(\overline{x}, .)(\HH)$ is an $(n-1)$-dimensional subspace $\HH_x$ of $\F_{2^n}$ that is unique to $x$.

    There are precisely $\Gauss{n}{1}$  points in $\HH$, and so there are $\Gauss{n}{1}$ many distinct $(n-1)$-dimensional subspaces $\HH_x$ of $\F_{2^n}$ to consider. 
    However, the number of $(n-1)$-dimensional subspaces of $\F_{2^n}$ is exactly $\Gauss{n}{1}$. So, we can interpret every $(n-1)$-dimensional subspaces of $\F_{2^n}$ as the image of $c_f(\overline{x}, .)$ for some $\overline{x}\in \HH$.
    In particular if one picks two elements of $\F_{2^n}^*$, they are contained in  some $(n-1)$-dimensional space $T$ of $\F_{2^n}$ and we can find a point $\overline{x}\in \HH$ such that the image of $c_f(\overline{x},.)$ is $T$.
\end{proof}

\begin{lem}\label{L: sigma linear}
 Let $f$ and $f'$ be quadratic crooked functions over $\F_{2^n}$ with $n>3$
       and let $\sigma$ be a permutation of $\F_{2^n}^*$ with $\sigma(0)=0$. Furthermore, let $\kappa$ be a collineation of $\PG(n,2)$ such that
    $ \sigma c_f=c_{f'}\kappa$. Then $\sigma$ is linear.
\end{lem}

\begin{proof}
Let $u,v\in \F_{2^n}^*$. In view of Lemma \ref{L: intersecting lines in H} there exist  $x,y,z\in \F_{2^n}^*$ such that $B_f(x,y)=c_f( (x,0),(y,0) )=u$ and $B_f(z,y)=c_f((z,0),(y,0))=v$.
Furthermore, we have seen that there is a linear operator $A$ such that $\kappa((x,0))=(Ax,0)$.
We have
    \begin{align*}
\sigma(u+v)&=\sigma(B_f(x,y)+B_f(z,y))\\
&=\sigma(B_f(x+z,y))\\
&=\sigma(c_f((x+z,0),(y,0) ))\\
&=c_{f'}((A(x+z),0),(Ay,0))\\
&=B_{f'}(A(x + z),Ay)\\
&=B_{f'}(Ax,Ay)+B_{f'}(Az,Ay)\\
&=c_{f'}((Ax,0),(Ay,0))+c_{f'}((Az,0),(Ay,0))\\
&=\sigma c_{f}((x,0),(y,0))+\sigma c_{f}((z,0),(y,0))\\
&=\sigma(u)+\sigma(v). 
\qedhere
\end{align*}
\end{proof}

We are now in a position to complete the proof of part 2 of Theorem \ref{T: main equivalence}.

\begin{prop}\label{P: }
    Let $f$ and $f'$ be quadratic crooked functions over $\F_{2^n}$ with $n > 3$. It holds that $\Pi_f$ and $\Pi_f'$ are equivalent if and only if there exist linear permutations $\sigma$ and $A$ of $\F_{2^n}$ and $\alpha\in\F_{2^n}$ such that $f'(x)=\sigma f(Ax + \alpha)+ \sigma(f(\alpha))$. 
\end{prop}

\begin{proof}
    Suppose that $\Pi_f$ and $\Pi_{f'}$ are equivalent. Lemmas \ref{L: K(H)=H} and \ref{L: sigma linear} imply that there exist a linear permutation $\sigma$ and a collineation $\kappa$ of $\PG(n, 2)$ which fixes $\mathcal{H}$ such that $\sigma c_{f'} = c_{f}\kappa$.
    Let $A\in \F_2^{n\times n}$, as well as $\alpha\in \F_2^{n}$, and $B\in \F_2^{1\times n}$, and $\beta \in \F_2$, such that we can understand $\kappa$ via the $(n+1)\times (n+1)$ matrix 
    $$\begin{pmatrix}
    A & \alpha \\
    B & \beta\\
\end{pmatrix}. $$
Recall also that $\kappa((x,x_1))$ is understood via 
$\kappa((x,x_1))=(Ax+\alpha x_1,Bx+\beta x_1 )$. As $\kappa(\HH)=\HH$ we have $\kappa((x,x_1))=(Ax+\alpha x_1,x_1 )$.\\
    Observe that $\sigma c_{f'}((x,1),(0,1))=\sigma f'(x)$ and $c_{f}\kappa((x, 1), (0, 1)) = f(Ax+\alpha)+f(\alpha)$. Thus,
    $$
    f'(x) = \sigma^{-1}f(Ax + \alpha) + \sigma^{-1}f(\alpha).
    $$

    Now suppose there exist $\sigma, A$ linear permutations of $\F_{2^n}$ and an element $\alpha \in \F_{2^n}$ such that $f'(x) = \sigma(f(Ax+ \alpha)) + \sigma(f(\alpha))$. Let $\kappa(x, x_1) = (Ax + \alpha x_1, x_1)$ be understood as a collineation of $\PG(n, 2)$. We will show that it follows that $\sigma^{-1} c_{f'}((x,x_1),(y,y_1)) = c_f\kappa((x,x_1),(y,y_1))$, thus implying that the parallelisms $\Pi_{f'}$ and $\Pi_f$ are equivalent. To do this, we consider two cases, either $(x_1, y_1)\in\{(0,1),(1,0),(1,1) \}$ or $(x_1, y_1) = (0, 0)$.

    Let $(x_1, y_1)\in\{(0,1),(1,0),(1,1) \}$. Since it must be the case that exactly two of the terms $x_1, y_1, x_1 + y_1$ are equal to 1, we demonstrate that for $(x_1, y_1) = (0, 1)$, we have $\sigma c{f'}((x, 0), (y, 1)) = c_f (\kappa (x, 0), \kappa(y, 1))$. The computations for the other two cases are identical.
    \begin{align*}
        \sigma^{-1} c_{f'}((x,0),(y,1))&=\sigma^{-1} f'(y)+\sigma^{-1} f'(x+y)\\
        &= f(Ay+\alpha)+f(\alpha)+f(A(x+y)+\alpha)+f(\alpha)\\
        &=f(Ay+\alpha)+f(Ax+Ay+\alpha)\\
        &=c_f((Ax,0),(Ay+\alpha,1)) \\
        &= c_f(\kappa(x, 0), \kappa(y, 1)).
    \end{align*}

    Now we consider the case $(x_1, y_1) = (0, 0)$. First, we observe that
    \begin{align*}
        \sigma^{-1}f'(x)&=f(Ax + \alpha) + f(\alpha) \\
        &= f(Ax + \alpha) + f(\alpha) + f(Ax) + f(Ax) \\
        &= f(Ax) + c_f((Ax, 0), (\alpha, 0)).
    \end{align*}
    Furthermore recall that for a quadratic function $f$ we have that $c_f((Ax,0),(\alpha,0))$  is the bilinear function $B_f(Ax,\alpha)$. Thus
     \begin{align*}
        \sigma^{-1}c_{f'}((x,0),(y,0))&=\sigma^{-1}f(x)+\sigma^{-1}f(y)+\sigma^{-1}f(x+y)\\
        &= f(Ax)+B_f(Ax,\alpha)+f(Ay)+B_f(Ay,\alpha)+f(A(x+y))+B_f(A(x+y),\alpha)\\
        &=c_f((Ax,0),(Ay,0))+B_f(Ax,\alpha)+B_f(Ay,\alpha)+B_f(Ax+Ay, \alpha).\\
        &=c_f((Ax,0),(Ay,0))+B_f(Ax+Ay,\alpha)+B_f(Ax+Ay, \alpha).\\
        &=c_f((Ax,0),(Ay,0)) \\
        &=c_f(\kappa(x, 0), \kappa(y, 0)).
        \qedhere
    \end{align*}
\end{proof}

\section{Concluding Remarks}

In \cite{godsil_crooked}, it was shown that a crooked function can be characterized by the Preparata-like code it defines as well as by the distance-regular graph it produces \cite{Bending1998, Distance}. We remark that this characterization does not seem to immediately extend to our definition of parallelisms via the coloring function $c_f$. In particular, going through the proof of Theorem \ref{T: c_f is parallelism}, we see that case 4 follows from a specialization of the definition of crookedness, and does not require the full strength of the definition of crookedness. 

In particular, it follows from our arguments in Section \ref{Section construction} that $f$ (over $\F_{2^n}$) defines a parallelism of $\PG(n, 2)$ via the coloring function $c_f$ if and only if $f$ is an APN permutation which satisfies $f(x) \neq f(x + y) + f(y) + f(x + z) + f(z)$ for all $x, y, z \in \F_{2^n}$, and $x \neq 0$. 

Let $H_x = \{D_x(y): y \in \F_{2^n}\}$ be the set of all derivatives in the direction of $x$. The following statements can be shown to be equivalent:

\begin{enumerate}
    \item  $c_f((x, x_1), (y, y_1)) = f(x + y) + f(x) + f(y) + f(x_1y + y_1x)$ defines a parallelism.
    \item  $f$ is an APN permutation which satisfies $f(x) \not \in H_x + H_x$ for all $x \neq 0$. 
    \item Let $F_{x, a} = \{ a , f(x) + a\}$ (a coset of $\{0, f(x) \}$ of the additive group of $\F_{2^n}$). Then $|F_{x, a} \cap H_x| = 1$ for all $x, a \in \F_{2^n}$.
\end{enumerate}

Since any quadratic APN permutation is crooked, we performed a search over known non-quadratic APN permutations. We did not find any functions which satisfy these more relaxed conditions on $f$. It is possible that the relaxed conditions in fact imply crookedness, but this does not seem to follow immediately.

 Finally, we note that that it may be possible to construct multiple inequivalent line-parallelisms starting with one quadratic function $f$.
 If $f'(x) = f(x) + L(x)$ where $L$ is a linear function such that $f'$ is EA-equivalent to $f$ but not affine equivalent, then Theorem \ref{T: main equivalence} implies that the corresponding parallelism $\Pi_{f'}$ is inequivalent to $\Pi_f$. 
 For crooked functions however, this seems to be a nontrivial matter.  A more detailed discussion about this can be found in \cite[Section 3]{crooked_2000}, where such a function $L$ is called a  \textit{linear accomplice} of $f$. 

\section*{Acknowledgments}

This paper was in part conceptualized during the 5th Pythagorean conference in Kalamata, the authors would like to thank the organizers for a wonderful conference.
The first author gratefully acknowledges the financial support provided by Stefan Witzel through the grant WI4079/6, which made it possible to attend the 5th Pythagorean Conference.
The authors would like to thank Sam Mattheus for bringing to their attention the paper \cite{Distance} which was the catalyst for this project.
The authors would like to thank Charlene Weiß for suggesting a nice design-theoretic proof of Lemma \ref{L: images} that is not in the current iteration of the paper.
We would also like to thank Alex Pott for many helpful discussions and for bringing the references \cite{Pott_APN_planarity, Pott2} to our attention.

\bibliographystyle{plain}
\bibliography{bibliography.bib}

\end{document}